\documentclass{article}
\usepackage{latexsym,amsmath,amssymb,amsthm,amsfonts, mathrsfs}
\usepackage[utf8]{inputenc}
\usepackage[T1]{fontenc}
\usepackage{color}

\newtheorem{theorem}{Theorem}

\newtheorem{lemma}[theorem]{Lemma}
\newtheorem{observation}{Observation}

\newtheorem{conjecture}{Conjecture}

\newcommand{\partiale}{\overrightarrow{\partial}_e}
\newcommand{\partialv}{\overrightarrow{\partial}_v}

\begin{document}

\title{The proofs of two directed paths conjectures of Bollob\'as and Leader}
\author{Trevor Pinto\thanks{Supported by an EPSRC doctoral studentship.}\\
\small School of Mathematical Sciences,\\[-0.8ex]
\small Queen Mary University of London,\\[-0.8ex] 
\small London E1 4NS, UK.\\
}
\maketitle

\begin{abstract}
Let $A$ and $B$ be disjoint sets, of size $2^k$,  of vertices of $Q_n$, the $n$-dimensional hypercube. In 1997, Bollob\'as and Leader proved that there must be $(n-k)2^k$ edge-disjoint paths between such $A$ and $B$. They conjectured that when $A$ is a down-set and $B$ is an up-set, these paths may be chosen to be directed (that is, the vertices in the path form a chain). We use a novel type of compression argument to prove stronger versions of these conjectures, namely that the largest number of edge-disjoint paths between a down-set $A$ and an up-set $B$ is the same as the largest number of  directed  edge-disjoint paths between $A$ and $B$. Bollob\'as and Leader made an analogous conjecture for vertex-disjoint paths and we prove a strengthening of this by similar methods. We also prove similar results for all other sizes of $A$ and $B$.
\end{abstract}

\section{Introduction}

The dimension $n$ hypercube, $Q_n$ is one of the most studied objects in combinatorics. It has vertex set $\mathcal{P}[n]$, the power set of $[n]=\{1,\dots n\}$, with an edge linking two vertices, $x$ and $y$ if $|x\triangle y|=1$. Equivalently, there is an edge between $x$ and $y$ if  $x=y\cup \{i\}$, or vice versa, for some $i$. Where convenient, we abbreviate the singleton set $\{i\}$ to $i$. The directed hypercube, $\overrightarrow{Q_n}$ is the directed graph formed by orienting all edges of $Q_n$ from $x$ to $x\cup i$. In other words, we direct all edges towards their larger endpoint. A \emph{directed path} in the hypercube is a path whose vertices form a chain. Equivalently, it is a path in the directed cube $\overrightarrow{Q_n}$.

\subsection{Edge disjoint paths in the cube}

The \emph{edge boundary} of a subset $S$ of $\mathcal{P}[n]$, is written $\partial_e(S)$, the set of $Q_n$-edges with exactly one endpoint in $S$. The \emph{directed edge boundary}, written $\partiale(S)$, is the set of edges in $\partial_e(S)$ with smaller endpoint in $S$.

The Edge Isoperimetric Inequality answers the extremal problem of which sets, of a given size, have smallest edge boundary. To state the theorem, we must define the \emph{binary order}: we let $x<y$ if $\max (x\triangle y)\in y$. Thus for all $k$, the subcube $\mathcal{P}[k]$ is an initial segment of the binary order on $\mathcal{P}[n]$. The Edge Isoperimetric Inequality, proved by Harper \cite{harperedge}, Lindsey \cite{lindsey}, Bernstein \cite{bernstein} and Hart \cite{hart} states that initial segments minimize the size of the edge boundary.

\begin{theorem}[Edge Isoperimetric Inequality]\label{edgeiso}
Let $A\subseteq \mathcal{P}[n]$. Let $I$ be the set of the first $|A|$ elements of $\mathcal{P}[n]$ in the binary order. Then $\left|\partial_e(A)\right|\geq |\partial_e(I)|$. In particular, if $|A|=2^k$, then its edge boundary is larger than that of a $k$-dimensional subcube; i.e. $|\partial_e(A)|\geq (n-k)2^n$.
\end{theorem}

We write $p_e(A,B)$ for the size of the largest collection of edge-disjoint paths between two disjoint subsets of the cube, $A$ and $B$. Similarly, we write $\overrightarrow{p}_e(A,B)$ for the size of the largest collection of edge-disjoint directed paths between disjoint $A$ and $B$. In 1997, Bollob\'as and Leader \cite{bollobasleader}, gave a lower bound on $p_e(A,B)$, in terms of $|A|$ and $|B|$.

\begin{theorem}[Bollob\'as-Leader \cite{bollobasleader}]\label{bledges}
Let $A$ and $B$ be disjoint subsets of $Q_n$, each of size $2^k$, for some non-negative integer $k$. Then there is a family of at least  $(n-k)2^k$  edge-disjoint directed paths from $A$ to $B$.
\end{theorem}

It is easy to see that this is best possible. Indeed, $p_e(A,B)$ is bounded above by $|\partial_e(A)|$  and when $A$ is a $k$-dimensional subcube, this is precisely $(n-k)2^k$.

Theorem \ref{bledges} is a special case of Bollob\'as and Leader's full result, which gives a lower bound for each pair of values of $|A|$ and $|B|$. This full result is stated in Section 3, together with some related discussion. For now, we write $BL_e(|A|,|B|)$ for the lower bound they gave for $p_e(A,B)$. This bound that they proved is not simply the minimum of the edge boundaries of initial segments of size $|A|$ and $|B|$, indeed, that is not a lower bound to $p_e(A,B)$.

A \emph{down-set} is a subset, $A$, of $Q_n$ such that if $x\in A$ and $y\subseteq x$, then $x\in Q_n$. An \emph{up-set} is the complement of a down-set.  Bollob\'as and Leader \cite{bollobasleader} asked if one can require the paths between up-sets and down-sets to be directed, and keep the same bounds. More precisely they proposed the following.

\begin{conjecture}
Let $A$, a down-set, and $B$, an up-set, be disjoint non-empty subsets of $\mathcal{P}[n]$. Then $\overrightarrow{p}_e(A,B)\geq BL_e(|A|,|B|)$. In particular, if $|A|=|B|=2^k,$ then $\overrightarrow{p}_e(A,B)\geq (n-k)2^k$.
\end{conjecture}

See also \cite{blproblempaper} for a brief description of their conjecture, submitted as an open problem to the British Combinatorial Conference.

In Section 2.1, we prove a strengthened version of the conjecture, that is essentially best possible:

\begin{theorem}\label{diredges}
Suppose $A$ and $B$ are disjoint subsets of $Q_n$, where $A$ is a down-set, and $B$ is an up-set. Then there are the same number of  edge-disjoint paths from $A$ to $B$ as edge-disjoint directed paths, i.e. $\overrightarrow{p}_e(A,B)= p_e(A,B)$.
\end{theorem}

Bollob\'as and Leader \cite{bollobasleader} use flow theorems (see for instance Chapter 3 of \cite{bollobas} for a good introduction to the topic) to demonstrate a relationship between edge-disjoint paths in the cube and edge-boundaries of subsets, and implicitly showed a directed version of this. More precisely they showed:

\begin{lemma}\label{edgelemma}
For all disjoint $non-empty$ subsets of $Q_n$, $A$ and $B$, $p_e(A,B)=\min\{|\partial_e(S)|: A\subseteq S\subseteq B^c\}$. If additionally $A$ is a down-set and $B$ is an up-set then $\overrightarrow{p}_e(A,B)=\min\{\partiale(S): A\subseteq S\subseteq B^c\}$.
\end{lemma}

We give Bollob\'as and Leader's proof of this lemma in Section 3.1.

Easily, this lemma allows us to deduce Theorem \ref{diredges} from the following directed version of the edge isoperimetric inequality, which we prove in Section 2.1 using an unusual compression argument. Roughly speaking, we define two different compression operators, neither of which always reduces the size of the directed edge boundary of a set, but we show that for each set at least one of them does.

\begin{theorem}\label{diredgeiso}
Let ${A}$ be an up-set and ${B}$ be a disjoint down-set, both non-empty subsets of $Q_n$. Then $\min\Big\{|\partiale(S)|: A\subseteq S\subseteq B^c\Big\}$ is attained by a down set. Thus $\min\Big\{\partiale(S):A\subseteq S\subseteq B^c\Big\}=\min\{\partial_e(S): A\subseteq S\subseteq B^c\}$.
\end{theorem}

\subsection{Vertex disjoint paths in the cube}

The \emph{vertex boundary} of $S$, written $\partial_v(S)$, is the set of vertices in $S^c$ adjacent to a vertex in $S$. In other words, $\partial_v(S)=\{x\in S^c: d(x,y)=1, \text{ for some } y\in S\}$, where $d(x,y)$ is the usual graph distance. The \emph{directed vertex boundary} of $S$, written $\partialv(S)$, is the set of vertices in $\partial_v(S)$ in $S^c$, with a smaller neighbour in $S$. 

The \emph{simplicial order} is defined by letting $x<y$ if either $|x|<|y|$ or if both $|x|=|y|$ and $x$ precedes $y$ in the lexicographic order, i.e. $\min(x\triangle y)\in x$. Note that for all $k$, the set $[n]^{(\leq k)}:=\{x\in \mathcal{P}[n]: |x|\leq k\}$ is an initial segment of simplicial order.

\begin{theorem}[Vertex Isoperimetric Inequality]\label{vertexiso}
Let $A\subseteq \mathcal{P}[n]$. Let $I$ be the set of the first $|A|$ vertices of $Q_n$ in the simplicial order. Then $|\partial_v(A)|\geq |\partial_v(I)|$.
\end{theorem}

We write $p_v(A,B)$ for the size of the largest collection of paths with vertex-disjoint interiors, between two disjoint subsets of the cube, $A$ and $B$. Similarly, we write $\overrightarrow{p}_v(A,B)$ for the size of the largest collection of directed paths between $A$ and $B$ that have vertex-disjoint interiors. Just as for the edge-disjoint case,  Bollob\'as and Leader, \cite{bollobasleader}, gave a lower bound on $p_v(A,B)$ in terms of $|A|$ and $|B|$. Their full theorem is given and discussed in Section 3.2, below is the special case Bollob\'as and Leader were most interested in.

\begin{theorem}[Bollob\'as-Leader \cite{bollobasleader}]\label{blvertices}
Let $A$ and $B$ be disjoint non-empty subsets of $Q_n$, with $|A|=|B|=\sum_{i=0}^k \binom{n}{k}$. Then $p_v(A,B)\geq \binom{n}{k+1}$. 
\end{theorem}

It is easy to see that this is essentially best possible, since if $A$ and $B$ are non-adjacent, every path from $A$ to $B$ must have one vertex in $\partial_v(A)$. If $A$ is the set of elements of weight at most $k$, sometimes written $[n]^{(\leq k)}$, then $\partial_v(A)$ is precisely the bound given in the Theorem.

We write $BL_v(|A|,|B|)$, for the lower bound given by Bollob\'as and Leader for $p_v(A,B)$. As in the edge case, this lower bound is not simply the isoperimetric bound- i.e. it is not the minimum of the vertex boundaries of initial segments of size $|A|$ and $|B|$. Indeed, they show that is not a lower bound for $p_v(A,B)$.

Bollob\'as and Leader also proposed a directed version of Theorem \ref{blvertices}, conjecturing that their bounds hold even for directed paths between up-sets and down-sets:

\begin{conjecture}
Let $A$ and $B$ be disjoint non-empty subsets of $Q_n$. Then $\overrightarrow{p}_v(A,B)\geq BL_v(|A|,|B|)$. In particular, if $|A|=|B|=\sum_{i=0}^k \binom{n}{k}$, then $\overrightarrow{p}_v(A,B)\geq \binom{n}{k+1}$.
\end{conjecture}

In Section 2.2 of this paper, we prove a strengthening of this conjecture.

\begin{theorem}\label{directedvertices}
Suppose $A$, a down-set,  and $B$, an up-set, are disjoint non-empty subsets of $Q_n$. Then $\overrightarrow{p}_v(A,B)=p_v(A,B)$.
\end{theorem}

As an intermediate step in the proof, we prove the following isoperimetric-type inequality, which may be of independent interest:

\begin{theorem}\label{dirvertexiso}
Let ${A}$ be a non-empty down-set and ${B}$ be a non-empty up-set, both non-empty subsets of $Q_n$. Suppose ${A}\subseteq S\subseteq B^c$, then there exists a down-set $S'$ satisfying $A\subseteq S' \subseteq B^c$ with $\partialv(S)\geq \partialv(S')$.
\end{theorem}

The proof of Theorem \ref{blvertices} uses a flow theorem (this time Menger's Theorem) to show a connection between the number of vertex-disjoint paths and vertex boundaries. Indeed, Bollob\'as and Leader show:

\begin{observation}\label{vertexobs}
The number of vertex-disjoint paths between $A$ and $B$ is equal to $e(A,B)$ plus the smallest vertex cut separating $A$ from $B$ in the graph $Q_n-E(A,B)$, i.e. the graph formed by deleting all edges from $A$ to $B$ from the hypercube.
\end{observation}

We use essentially the same approach to show  a directed version of this observation, which we use to deduce Theorem \ref{directedvertices} from Theorem \ref{dirvertexiso}.

It is interesting to note that although the Edge Isoperimetric Inequality and the Vertex Isoperimetric Inequality use different approaches, our two directed versions have a very similar proof, both relying on the same compressions. Again, neither of these compressions works on its own, but we show at least one of them works for each set.

\section{Directed Isoperimetric Inequalities and Directed Paths}

We introduce here the two different classes of compression, which we use to prove Theorems \ref{diredgeiso} and \ref{dirvertexiso}. Each of these compressions makes $S$ more like a down set, in some sense that we will make concrete. For $S\subseteq \mathcal{P}[n]$, and $i\in [n]$, we say that:
\[C_i(S)=\{x\in S: x\setminus \{i\} \in S\} \text{ and } D_i(S)=S\cup \{x:x\cup\{i\}\in S\}.\]

We first state some properties of these compressions that will be used to prove both edge and vertex versions of our theorems.

\begin{observation}\label{containments}
If $A$ is a down-set, $B$ is an up-set, and $A\subseteq S \subseteq B^c$, then $A\subseteq  C_i(S) \subseteq S \subseteq D_i(S)\subseteq B^c$.
\end{observation}

We call a set $S$ \emph{i-down} if $x\in S \Rightarrow x\setminus i \in S$. Clearly $S$ is a down set if and only if $S$ is $i$-down for all $i$. It is easy to see that both $C_i(S)$ and $D_i(S)$ are $i$-down sets. The following lemma shows that the operators $C_i$ and $D_i$ preserve the $j$-down property.

\begin{lemma}\label{preservesdownness}
Let $S\subseteq Q_n$ and $i, j \in [n]$. If $S$ is $j$-down then so is $D_i(S)$ and $C_i(S)$.
\end{lemma}

\begin{proof}
If $i=j$, this is trivial, so we assume otherwise.

Suppose $x\in C_i(S)$, then we must have $x\setminus i\in C_i(S)$, as $C_i(S)$ is $i$-down. Since $C_i(S)\subseteq S$, we get that $x$ and $x\setminus i$ are in $S$. By our assumption that $S$ is $j$-down, this implies $x\setminus j$ and $x\setminus \{i, j\}\in S$. The definition of $C_i$ allows us to conclude that $x\setminus j \in C_i(S)$, as required. 

Suppose now that $x\in D_i(S)$. Suppose first that $x\in S$ then $x\setminus j\in S$, since $S$ is $j$-down. This implies that $x\setminus j \in D_i(S)$, as $S\subseteq D_i(S)$. If instead $x\notin S$, then $x\cup i\in S$, by the definition of $D_i$. This implies that $(x\cup i)\setminus j\in S$ and thus $x\setminus j \in D_i(S)$, again by the definition of $D_i$.
\end{proof}

For $S\subseteq \mathcal{P}[n]$, the \emph{$i$-sections} of $S$ are the sets $S^+_i:=\{x\in \mathcal{P}([n]\setminus \{i\}): x\cup \{i\}\in S\}$ and $S^-_i:=\{x\in \mathcal{P}([n]\setminus {i}): x\in S\}$.

We also define, for $S\subseteq \mathcal{P}[n]$ and $i\in [n]$, three related subsets $T,U,V, \subseteq \mathcal{P}([n]\setminus \{i\})$:
\begin{align*}
T=&T_{S,i}=S_i^+\cap S_i^-=\{x\in \mathcal{P}([n]\setminus \{i\}): x\in S \text{ and } x\cup i\in S\}\\
U=&U_{S,i}=S_i^-\setminus S_i^+=\{x\in \mathcal{P}([n]\setminus \{i\}): x\in S \text{ and } x\cup i\notin S\}\\
V=&V_{S,i}=S_i^+\setminus S_i^-=\{x\in \mathcal{P}([n]\setminus \{i\}): x\notin S \text{ and } x\cup i \in S\}\\
W=&W_{S,i}=\{x\in \mathcal{P}([n]\setminus \{i\}): x\notin S, x\cup \{i\}\notin S\}
\end{align*}

Given a subset $A$ of $\mathcal{P}([n]\setminus \{i\})$, we write $A\times \{i\}$ for the set $\{a\cup\{i\}: a\in A\}$.

These sets give us another way of viewing $C_i$ and $D_i$. Indeed, $S=(T\cup U)\cup (T\cup V)\times\{i\}$. Similarly, $C_i(S)=S\setminus(V\times \{i\})$ and $D_i(S)=S\cup U$.

\subsection{Edge Version}

For sets, $S_1, S_2\subseteq V(Q_n)$, we denote by $\partial_e(S_1,S_2)$ the set of edges with one endpoint in $S_1$ and one endpoint in $S_2$. Similarly, we write $\partiale(S_1,S_2)$ for the set of edges with the smaller endpoint in $S_1$ and the larger endpoint in $S_2$. One can see that $\partial_e(S)=\partial_e(S, S^c)$ and $\partiale(S)=\partiale(S,S^c)$.

\begin{proof}[Proof of Theorem \ref{diredgeiso}]

The majority of the proof is contained in the following key lemma.

\begin{lemma}\label{compressionreducesoutedges}
For any set $S$, and all $i$, $|\partiale(S)|\geq \min\Big\{|\partiale(D_i (S))|, |\partiale(C_i(S))|\Big\}$.
\end{lemma}

\begin{proof}[Proof of Lemma \ref{compressionreducesoutedges}]
For convenience, in the proof of this lemma, we write $D$ for $D_i(S)$ and $C$ for $C_i(S)$. It is easy to see that the contribution to $\partiale(S)$ from edges along the $i$ direction, is exactly the same as the contribution to $\partiale(C)$ and to $\partiale(D)$.

Firstly, we can see that since $D$ is a superset of $S$, any element of the directed edge boundary of $S$ is in the directed edge boundary of $D$ unless its larger endpoint is in $D\setminus S$. Thus $\partiale(S)\setminus\partiale(D)=\partiale(T,V)$.

Conversely, an element of $\partiale(D)$ is an element of $\partiale(S)$ unless its smaller endpoint is in $D\setminus S$. Therefore, $\partiale(D)\setminus \partiale(S)=\partiale(V, W)$.

Similar arguments show that $\partiale(S)\setminus\partiale(C)=\partiale(V\times \{i\},(W\cup U)\times \{i\})$ and that $\partiale(C)\setminus \partiale(S)=\partiale(T\times \{i\}, V\times\{i\}).$

Thus \begin{align*}
\big|\partiale(S)\big|-\big|\partiale(D)\big|&=\big|\partiale(T,V)\big|-\big|\partiale(V, W)\big|\\
&\geq  \big|\partiale(T,V)\big|- \big|\partiale(V,W\cup U)\big|\\
&=\big|\partiale(C)\big|-\big|\partiale(S)\big|\\
\end{align*}

Therefore, $|\partiale(S)|\geq \frac{1}{2}\left(|\partiale(D)|+|\partiale(C)|\right)$, which concludes the proof of the Lemma.

\end{proof}

Given $S$ such that $A\subseteq S\subseteq B^c$, we use Lemma \ref{compressionreducesoutedges} and, for $i=1,...,n$, successively apply either $D_i$ or $C_i$ and to a set $S'$ with $|\partiale(S')|\leq |\partiale(S)|$. By Lemma \ref{preservesdownness}, $S'$ must be $i$-down for all $i$ and thus is a down-set. By Observation \ref{containments}, we have that the set satisfies the required containments. Since, the directed edge boundary of a down-set is the same as the edge boundary, Theorem \ref{bledges}$'$ finishes the proof.

\end{proof}

\subsection{Vertex version}

The proof of Theorem \ref{dirvertexiso} is very similar to the proof of the edge version, but with a slightly different calculation. 

\begin{proof}[Proof of Theorem \ref{dirvertexiso}]

Again, the bulk of the proof is in the following lemma.

\begin{lemma}\label{compressionreducesupneighbours}
For any set $S$, and all $i$, $\big|\partiale(S)\big|\geq \min\Big\{\big|\partiale(D_i (S))\big|, \big|\partiale(C_i(S))\big|\Big\}$.
\end{lemma}

\begin{proof}
Once more, we write $C$ for $C_i(S)$ and $D$ for $D_i(S)$. Additionally, we write $h(S)=S\cup \partialv(S)$.

Since $C$ is a subset of $S$, any vertex in the directed vertex boundary of $C$ is in the directed vertex boundary of $S$ unless it is in $C\setminus S$. Thus $\partialv(C)\setminus \partialv(S)=\left(\partialv(T)\cap V\right)\times\{i\}$.

On the other hand, any vertex in $\partialv(S)$ but not in $\partialv(C)$ must neighbour a vertex in $S\setminus C$ and thus  $\partialv(S)\setminus\partialv(C)=\partialv(V)\times\{i\}\setminus\partialv(C)$. Since the set of vertices in $\partialv(C)$ that contain $i$ is $(U\times\{i\})\cup h(T\times\{i\})$, we may conclude that $\partialv(S)\setminus\partialv(C)=\left(\partialv(V)\setminus (h(T)\cup U)\right)\times \{i\}$.

Similarly, we have that $\partialv(S)\setminus \partialv(D)= \partialv(T\cup U)\cap V$ and that $\partialv(D)\setminus \partialv(S)=\partialv(V)\setminus h(T\cup U)$.

Since $V$ is disjoint from $T\cup U$, we see that $\partialv(T)\cap V\subseteq \partialv(T\cup U)\cap V$. Thus
\begin{align*}
\big|\partialv(S)\big|-\big|\partialv(D)\big|&=\big|\partialv(T\cup U)\cap V\big| -\big|\partialv(V)\setminus h(T\cup U)\big|\\
&\geq \big|\partialv(T)\cap V\big|- \big|\partialv(V)\setminus (h(T)\cup U)\big|\\
&=\big|\partialv(C)\big|-\big|\partialv(S)\big|.
\end{align*}

Therefore, $\big|\partialv(S)\big|\geq \frac{1}{2} \left( \big|\partialv(D)\big|+\big|\partialv(C)\big|\right)$.

\end{proof}

Given $S$, we can use Lemma \ref{compressionreducesupneighbours}, for $i=1,...,n$, and successively apply either $D_i$ or $C_i$ to yield a set $S'$ with $\big|\partialv(S')\big|\leq \big|\partialv(S)\big|$. By Lemma \ref{preservesdownness}, $S'$ must be $i$-down for all $i$ and thus is a down-set. By Observation \ref{containments}, we have that $S'$ satisfies the required containments. The final part of the theorem follows from properties of the $b$ function.  

\end{proof}

We now deduce Theorem \ref{directedvertices}, on vertex-disjoint directed paths.

\begin{proof}[Proof of Theorem \ref{directedvertices}]

Let $F=\{xy\in E(Q_n): x\in A, y\in B\}.$ We apply the directed version of Menger's Theorem to the directed graph $G=\overrightarrow{Q_n}-F$. It tells us that the number of paths in $G$ from $A$ to $B$, with vertex-disjoint interiors is the same as the minimum vertex cut separating $A$ from $B$ in $G$. This is the same as $\min\Big\{\big|\partialv(S)\big|:A\subseteq S\subseteq B^c\Big\}-|\{x\in B: d(x,y)=1, $ for some $y\in A\}|$. Theorem \ref{dirvertexiso} implies this is the same as $\min\Big\{\big|\partial_v(S)\big|:A\subseteq S\subseteq B^c\Big\}-|\{x\in B: d(x,y)=1, $ for some $y\in A\}|$, since the directed boundary is minimized by a down-set.  Observation \ref{vertexobs} concludes the proof.
\end{proof}

\section{Bollob\'as and Leader's Theorems}

For completeness, in this section we discuss the full versions of Bollob\'as and Leader's theorems on edge-disjoint and vertex-disjoint paths.

\subsection{Edge-disjoint paths}

In this subsection, we state Bollob\'as and Leader's full version of Theorem \ref{bledges}, give its proof, as well as that of Lemma \ref{edgelemma}, upon which it relies. 

First we give an approximation to the size of the edge boundary of an initial segment of binary order. Chung, F\"uredi, Graham and Seymour \cite{chungetal} observed that a good lower bound for $|\partial_e(I)|$, the size of the edge boundary of the initial segment of binary of size $x$ is:

\[
e(x)=e_n(x)=\begin{cases} x(n-\log_2 x) & \text{if } x\leq 2^{n-1},\\ (2^n-x)(n-\log_2(2^n-x)) & \text{if } x>2^{n-1}.\end{cases}
\]

This function $e(x)$ is easier to work with than $|\partial_e(I)|$, as there is a greater degree of monotonicity, and plays a key role in the proof of the following theorem. Note that $e(2^k)=(n-k)2^k$, which tells us that the following theorem is a generalisation of Theorem \ref{bledges}.

\begin{theorem}[Bollob\'as-Leader \cite{bollobasleader}]\label{bledgesfull}
Let $A$ and $B$ be disjoint non-empty subsets of $Q_n$. Then there is a family of at least  $\min\left\{e(|A|),e(|B|),2^{n-1}\right\}$  edge-disjoint directed paths from $A$ to $B$.
\end{theorem}

In \cite{bollobasleader}, the bound in the Theorem was stated slightly incorrectly in the case where both sets have size very close to $2^{k-1}$; the version stated above is the amended version.

The function $e$ is monotone increasing up to $x=2^n/\exp(1)$, it then decreases until $x=2^{n-1}$, and is symmetric about this point. Although the argument of \cite{bollobasleader} is essentially correct, it was incorrectly stated that $e$ is increasing up to $x=2^{n-1}$, leading to the erroneous bound $\min\{e(|A|), e(|B|)\}$ in the Theorem.

We give the proof of Lemma \ref{edgelemma} and then deduce Theorem \ref{bledgesfull}. Note that only the first part of the Lemma is required for Theorem \ref{bledgesfull}, but the second part was required for our Theorem \ref{diredges}, so we prove this part in the greater detail.

\begin{proof}[Proof of Lemma \ref{edgelemma}]
To prove the second part, we use the Max-Flow-Min-Cut Theorem on the directed graph $\overrightarrow{Q_n}$ with each edge having capacity 1 and all elements of $A$ regarded as sources and all elements of $B$ as sinks. In this setup, the value of the maximum flow is the same as the number of edge-disjoint directed paths from $A$ to $B$, since the Integrality Theorem implies there is an integer-valued maximum flow. The Max-Flow-Min-Cut Theorem states that the maximum flow value is also equal to the capacity of the smallest edge cut. Given an edge cut, write $S$ for the component containing $A$ in the graph formed by deleting  the edge cut. Clearly $A\subseteq S \subseteq B^c$. If the cut is minimal, then $\partiale(S)$ is the whole cut and its capacity is precisely $\big|\partiale(S)\big|$.

Similarly, to prove the first part, we apply the Max-Flow-Min-Cut theorem to the graph $Q_n$, giving each edge capacity 1, and viewing $A$ as the set of sources and $B$ as the set of sinks. This time, the maximum flow is simply $p_e(A,B)$, and as before, we may show that the minimum edge cut must be $\partial_e(S)$ for some $S$ satisfying $A\subseteq S\subseteq B^c$, and that for all such S, $\partial_e(S)$ is an edge cut, which concludes the proof.
\end{proof}

\begin{proof}[Proof of Theorem \ref{bledgesfull}]
By Lemma \ref{edgelemma}, we may choos $S$ with $p_e(A)=\partial_e(S)$ and $A\subseteq S\subseteq B^c\}$. Recall that $|\partial_e(S)|\geq e(|S|)$. If $e(|S|)\geq 2^{n-1}$, we are done. If not, since $e(2^{n-2})=e(2^{n-1})=e(3\cdot 2^{n-2})=2^{n-1}$, we have that $|A|\leq |S|<2^{n-2}$ or $|B^c|\leq |S^c|<2^{n-2}.$ In either case, monotonicity of $e$ in these intervals completes the proof. 
\end{proof}

\subsection{Vertex-disjoint paths and matchings.}

In this section we give the full version of Bollob\'as and Leader's lower bound on $p_v(A,B)$, for all values of $|A|$ and $|B|$.   

We first follow \cite{bollobasleader} and define the function $b(x)$, used as a lower bound for $|\partial_v(S)|$, where $S$ is a set of size $x$. For all $n$, we may write any smaller positive number $x$ uniquely in the form $\sum_{i=0}^k \binom{n}{k} +\alpha \binom{n}{k+1}$, for some $0\leq k\leq n$ and $0\leq \alpha <1$. We then define
\[b(x)=b_n(x)=(1-\alpha)\binom{n}{k+1}+\alpha \binom{n}{k+2}.\]

This allows us to state the full vertex-disjoint paths result from \cite{bollobasleader}.

\begin{theorem}[Bollob\'as-Leader \cite{bollobasleader}]\label{blfullvertices}
Let $A$ and $B$ be disjoint non-empty subsets of $Q_n$. Then there is a family of at least the minimum of $b(|A|)$ and $b(|B|)$ paths from $A$ to $B$ with vertex-disjoint interiors. 
\end{theorem}

Note that $b(\sum_{i=0}^k \binom{n}{k})=\binom{n}{k+1}$, so this theorem agrees with the special case, Theorem \ref{blvertices}, stated above.

In the case where $n$ is even and $|A|$ and $|B|$ are very close to $2^{n-1}$, the proof in \cite{bollobasleader} contains a small error in a calculation, although the theorem is correct as stated. For completeness, we give the full, amended proof here, despite the change being a minor one. Indeed, the change is simply using a slightly stronger lower bound to $\partialv$ than $b(x)$.

The function $b(x)$ is increasing up to $x=\sum_{i=1}^{\lceil n/2\rceil} \binom{n}{i}$, and is decreasing thereafter. If $n$ is odd, this is equal to $2^{n-1}$. If $n$ is even, however, this is slightly less than $2^{n-1}$. It was incorrectly stated in \cite{bollobasleader} that $b$ is increasing up to $2^{n-1}$ in both cases.

For the amended proof, we will require a weak version of the Kruskal--Katona Theorem, due to Lov\'asz \cite{lovasz}. Here, and in what follows, we write $\overleftarrow{\partial}_v A:=\{x\setminus\{i\}: i\in x \text{ and } x\in A\}$. This is sometimes known as the \emph{lower shadow of $A$}.

\begin{theorem}\label{weak KK}
Let $A\subseteq [n]^{(r)}$. Write $|A|=\binom{x}{r}$, $x\in \mathbb{R}, x>r-1$. Then $|\overleftarrow{\partial}_v A|\geq \binom{x}{r-1}$.
\end{theorem}

The proof of Theorem \ref{blfullvertices} is somewhat analogous to that of Lemma \ref{edgelemma}, but is complicated slightly by the fact there may be some edges from $A$ to $B$, so we cannot directly apply flow theorems in $Q_n$.

\begin{proof}[Proof of Theorem \ref{blfullvertices}]
We let $F=\{xy\in E(Q_n): x\in A, y\in B\}$ and we will apply Menger's theorem in the graph $G=Q_n-F$. Writing $A_1$ for $\{x\in A:xy\in F \text{ for some } y \in B\}$, and similarly $B_1$ for $\{x\in B:xy\in F \text{ for some } y \in A\}$, it is clear that $|F|\geq \max (|A_1|,|B_1|).$ It is therefore sufficient to show that any set $C\subseteq Q_n$ separating $A$ from $B$ in $G$ has size at least $\min(b(|A|), b(|B|))-\max(|A_1|,|B_1|).$ 

Let $C$ be a subset of $Q_n$ that separates $A$ from $B$ in $G$. Let $A'$ be the union of the components of vertices of $A$ in the graph $G-C$, and define $B'$ to the the union of all other components, i.e. $B=V(Q_n)\setminus (A'\cup S)$. We may assume that $|A'|\leq |B'|$. Since $A'$ and $B'$ are disjoint, we get that $|A'|\leq 2^{n-1}$. If $n$ is even  and $|A'|\leq \sum_{i=0}^{n/2-1}\binom{n}{i}$ or if $n$ is odd, then $\partial_v (A')\geq b(|A|)$, by the monotonicity of $b$ up to this point. Since also $\partial_v (A')\subseteq S\cup B_1$,  we are done in this case.

In the other case, $n$ is even, $|A'|=\sum_{i=0}^{n/2-1} \binom{n}{i}+\alpha \binom{n}{n/2}$ and  $|B'|=\sum_{i=0}^{n/2-1} \binom{n}{i}+\beta \binom{n}{n/2}$, for some $0<\alpha, \beta <1$. Since $A'$ and $B'$ are disjoint, $\alpha\leq 1/2$.

Recall that  $|\partial_v A'|\geq |\partial_v I|$, where $I$ is an initial segment of the simplicial order, with $|I|=\sum_{i=0}^{n/2-1} \binom{n}{i}+\alpha \binom{n}{n/2}$. We write $I_0$ for $|I\cap [n]^{(n/2)}|$.It is easy to see that $|I_0|=\alpha\binom{n}{n/2}$ and that $\big|\partialv(I)\big|=(1-\alpha)\binom{n}{n/2}+\big|\partialv(I_0)\big|$.  We will show that $\big|\partialv I_0\big|\geq |I_0|$. Let $J_0=\{x^c: x\in I_0\}$, a subset of $[n]^{(n/2)}$. Note that $\big|\overleftarrow{\partial}_v(J_0)\big|=\big|\partialv I_0\big|$. Choose $x$ such that $|J_0|=\binom{x}{n/2}$. Since $\binom{n-1}{n/2}=\frac{1}{2}\binom{n}{n/2}\geq |J_0|$, we have that $x\leq n-1$.

Thus, by Theorem \ref{weak KK} we have:
\begin{align*}
\big|\overleftarrow{\partial}_v (J_0)\big| -|J_0|&\geq \binom{x}{n/2-1}-\binom{x}{n/2}\\
&\frac{x(x-1)\cdot \dots \cdot (x-n/2+2)}{(n/2)!}[ n/2-(x-n/2+1)]\\
&\geq 0.
\end{align*}

This implies that $\big|\partialv (I_0)\big|\geq |I_0|$, and so $|\partial A'|\geq \binom{n}{n/2}$. Since $B_1=B'\cap \partial A'$, we have that $|B_1|\geq |A'|+|\partial A'| +|B'|-2^n\geq (\alpha+\beta)\binom{n}{n/2}$. Since $|S|=(1-\alpha-\beta)\binom{n}{n/2}$ and $b(|A|)\leq \binom{n}{n/2}$, we are done.

\end{proof}

Note that the proof implies Observation \ref{vertexobs} with no extra work.

Essentially the same issue occurs in the proof of another theorem from \cite{bollobasleader}. For completeness, we state the theorem below and give an amended proof. 

The \emph{surface} of $S\subseteq \mathcal{P}[n]$, written $\sigma(S)$, is the set of vertices of $S$ adjacent to a vertex in $S^c$. In other words, $\sigma(S)=\{x\in S: \exists y\in S^c, d(x,y)=1\}$. The reader may notice a similarity to the definition of $\partial_v(S)$. Indeed, $\sigma(S)=\partial_v(S^c)$. 

We write $s(x)=s_n(x)=(1-\alpha)\binom{n}{k}+\alpha\binom{n}{k+1}$, where $x=\sum_{i=0}^k \binom{n}{k}+\alpha \binom{n}{k+1}$ for some $\alpha\in [0,1)$. The relationship between $\sigma$ and $\partial_v$ implies that $\sigma(A)\geq s(|A|)$, for all sets $A$.

Bollob\'as and Leader showed the following, essentially best possible, bound on the size of matchings between two complementary sets, in terms of the size of the smaller set.

\begin{theorem}[Bollob\'as-Leader \cite{bollobasleader}]\label{matchings}
Let $A$ be a subset of $Q_n$ with $|A|\leq 2^{n-1}$. Then there is a matching from $A$ to $A^c$ of size at least $s(|A|)$.
\end{theorem}

Again, the theorem was stated correctly, but in \cite{bollobasleader} it was incorrectly claimed that $b(x)$ is increasing up to $x=2^{n-1}$. Once more, the fix is a small part of the proof, the rest comes directly from \cite{bollobasleader}.

\begin{proof}[Proof of Theorem \ref{matchings}]

By the defect form of Hall's Marriage Theorem, there is a matching of size $s(|A|)$ if there is no $B\subset A$ with $|\partial B\cap A^c|< |B|-(|A|-s(|A|))$. Suppose such a $B$ existed. Then we must have that: $|B|\geq |A|-s(|A|)$.

If $|B|\leq \sum_{i=0}^{\lfloor n/2\rfloor}\binom{n}{i}$, then by monotonicity of $b$ up to this point, $b(|B|)\geq b(|A|-s(|A|)).$ Note that $b(|A|-s(|A|))=s(|A|)$, by definition of $b$ and $s$, so $b(|B|)\geq s(|A|)$.  Otherwise,  the assumption on the size of $|A|$ implies that $n$ is even and that $|A|=\sum_{i=0}^{ n/2}\binom{n}{i}+ \alpha \binom{n}{n/2}$ and $|B|=\sum_{i=0}^{ n/2}\binom{n}{i}+ \beta \binom{n}{n/2}$, for some $0<\beta\leq \alpha  \leq 1/2$. Thus,
 \[b(|B|)=(1-\beta)\binom{n}{n/2}+\beta\binom{n}{n/2+1}, \quad s(|A|)=(1- \alpha)\binom{n}{n/2-1}+ \alpha\binom{n}{n/2}.
\]

Therefore \[ b(|B|)-s(|A|)=(1-\beta -\alpha)\binom{n}{n/2}-(1-\beta-\alpha)\binom{n}{n/2+1}.\]

Hence in this case, we also have $b(|B|)\geq s(|A|)$. This implies that $|\partial_v B \cap A^c|\geq s(|A|) -|A\setminus B|$, concluding the proof.
\end{proof}

\section*{Acknowledgements}
The author would like to thank Robert Johnson and David Ellis for useful advice and helpful discussions relating to the topic and Imre Leader for helpful comments that greatly improved the presentation of this paper.

\end{document}